\documentclass[a4paper,10pt]{article}
\usepackage{latexsym,amsmath,amsthm,amssymb}
\usepackage{a4wide}
\usepackage{hyperref}
\usepackage{marginnote}
\usepackage{color}
\usepackage{cite}
\hypersetup{
pdftitle={CKN--manifold}   
pdfauthor={Van Hoang},
colorlinks = true,
linkcolor = magenta,
citecolor = blue,
}

\theoremstyle{plain}
\newtheorem{theorem}{Theorem}[section]

\newtheorem{lemma}[theorem]{Lemma}
\newtheorem{corollary}[theorem]{Corollary}

\theoremstyle{definition}

\theoremstyle{remark}

\renewcommand{\thefootnote}{\arabic{footnote}}

\def\R{\mathbb R}

\def\al{\alpha}
\def\om{\omega}
\def\Om{\Omega}
\def\be{\beta}
\def\de{\delta}
\def\De{\Delta} 

\def\si{\sigma}
\def\vphi{\varphi}
\def\ep{\epsilon}
\def\na{\nabla}
\def\pa{\partial}
\def\la{\langle} 
\def\ra{\rangle} 
\def\lt{\left}
\def\rt{\right}

\def\i0i{\int_0^\infty}

\def\bD{\text{\bf D}}

\def\B{\mathbb B}

\numberwithin{equation}{section}

\title{Sharp Hardy and Rellich type inequalities on Cartan--Hadamard manifolds and their improvements}
\author{Van Hoang Nguyen
\footnote{Institut de Math\'ematiques de Toulouse, Universit\'e Paul Sabatier, 118 Route de Narbonne, 31062 Toulouse c\'edex 09, France.}
}

\begin{document}
\maketitle


\renewcommand{\thefootnote}{}

\footnote{Email: \href{mailto: Van Hoang Nguyen <van-hoang.nguyen@math.univ-toulouse.fr>}{van-hoang.nguyen@math.univ-toulouse.fr}, \href{mailto: Van Hoang Nguyen <vanhoang0610@yahoo.com>}{vanhoang0610@yahoo.com}}

\footnote{2010 \emph{Mathematics Subject Classification\text}: 26D10, 31C12, 53C20, 53C21.}

\footnote{\emph{Key words and phrases\text}: Hardy inequality, Rellich inequality, critical Hardy inequality, critical Rellich inequality, Cartan--Hadamard manifolds, sharp constant}

\renewcommand{\thefootnote}{\arabic{footnote}}
\setcounter{footnote}{0}

\begin{abstract}
In this paper, we prove several new Hardy type inequalities (such as the weighted Hardy inequality, weighted Rellich inequality, critical Hardy inequality and critical Rellich inequality) for radial derivations (i.e., the derivation along the geodesic curve) on Cartan--Hadamard manifolds. By Gauss lemma, our new Hardy inequality are stronger than the classical one. We also established the improvements of these inequalities in terms of sectional curvature of underlying manifolds which illustrate the effect of curvature to these inequalities. Furthermore, we obtain some improvements of Hardy and Rellich inequality in hyperbolic space $\mathbb H^n$. Especially, we show that our new Rellich inequality is indeed stronger the classical one in hyperbolic space $\mathbb H^n$.
\end{abstract}

\section{Introduction}
The motivation of this paper is to prove some new Hardy and Rellich type inequalities on Cartan--Hadamard manifolds (i.e., complete, simply connected Riemannian manifolds with non-positive sectional curvature). All our obtained inequalities are in sharp forms. They are stronger than and generalize several classical inequalities in the Euclidean space $\R^n$. Let $n\geq 2$ and $p \in (1,n)$, the classical Hardy inequality in $\R^n$ states that
\begin{equation}\label{eq:HardyRn}
\lt(\frac{n-p}p\rt)^p \int_{\R^n} \frac{|f|^p}{|x|^{p}} dx \leq \int_{\R^n} |\na f|^p dx,\qquad f\in C_0^\infty(\R^n).
\end{equation}
The constant $(n-p)^p/p^p$ in \eqref{eq:HardyRn} is sharp. A similar inequality with the same best constant also holds if $\R^n$ is replaced by any domain $\Om$ containing the origin. The Hardy inequality \eqref{eq:HardyRn} plays important role in several branches of mathematics such as the partial differential equations, spectral theory, geometry etc. We refer the reader to \cite{BEL15,BM97,Davies1998} for reviews of this subject and to \cite{BV97,FT2002} for the improvements of this inequality when $\R^n$ is replaced by the bounded domains containing the origin. 

In the critical case $p=n$, the inequality \eqref{eq:HardyRn} fails for any constant. However, the following inequality holds in the unit ball $B$ with center at origin of $\R^n$
\begin{equation}\label{eq:criticalET}
\int_B |\na f|^n dx \geq \frac{(n-1)^n}{n^n} \int_B \frac{|f|^n}{|x|^n \lt(1+ \ln \frac1{|x|}\rt)^n} dx,\quad f\in C_0^\infty(B),
\end{equation}
for any $n\geq 2$ (see, e.g., \cite{ET99}). The constant $(n-1)^n/n^n$ is the best constant in \eqref{eq:criticalET}. It also was shown in \cite{ET99} that \eqref{eq:criticalET} is equivalent to the critical case of the Sobolev--Lorentz inequality. It is remarked that \eqref{eq:criticalET} is not invariant under the scaling as \eqref{eq:HardyRn}. A scaling invariant version of \eqref{eq:criticalET} (nowaday called the critical Hardy inequality) was recently established by Ioku and Ishiwata \cite{II2015},
\begin{equation}\label{eq:criticalII}
\frac{(n-1)^n}{n^n} \int_B \frac{|f|^n}{|x|^n \lt(\ln \frac1{|x|}\rt)^n} dx \leq \int_B \lt|\frac x{|x|} \cdot \na f(x)\rt|^n dx, \quad f\in C_0^\infty(B).
\end{equation}
Again, the constant $(n-1)^n/n^n$ is sharp. The inequality \eqref{eq:criticalII} in bounded domains was also discussed in \cite{II2015}. It is surprise that the critical Hardy inequality \eqref{eq:criticalII} is equivalent to the Hardy inequality \eqref{eq:HardyRn} in larger dimension spaces (see \cite{ST2017}). We refer the readers to \cite{IIO2016} for a global scaling invariant version of \eqref{eq:criticalII} which we do not mention here.

Recently, there is an enourmous work to generalize the Hardy inequality to many different settings. For examples, the fractional Hardy inequality was established in \cite{FMT2013,Yafaev,Nguyen2016,FLS2008,FS2008} and references therein. The Hardy inequality also was proved on group structure, e.g., on Heisenberg groups \cite{GL,D'ambrosio,NZW,DGP}, on polarisable groups \cite{GK}, on Carnot groups \cite{JS,Lian}, on stratified groups \cite{CCR,RS2017carnot}, and on more general homogeneous groups \cite{RS2017,RS2016,RS2017L2}. The Hardy inequality on Riemannian manifold $(M,g)$ was studied by Carron \cite{Carron} in the weighted $L^2-$form under some geometric assumption on the weighted function $\rho$. More precisely, he proved the following inequality
\[
\frac{(\gamma -1)^2}{4}\int_M \frac{u^2}{\rho^2} dV_g \leq \int_M |\na_g u|_g^2 dV_g,\quad u\in C_0^\infty(M),
\]
where $\rho$ is a nonnegative function on $M$ such that $|\na_g \rho|_g =1$, $\De _g\rho \geq \gamma/\rho$, and $dV_g, \na_g$, $\De_g$ and $|\cdot|_g$ denote the volume element, gradient, Laplace--Beltrami operator and the length of a vector field with respect to the Riemannian metric $g$ on $M$ respectively, and the set $\rho^{-1}(0)$ is a compact set of zero capacity. Under the same hypotheses on the function $\rho$, Kombe and \"Ozaydin \cite{KO2009} extended the result of Carron to the case $p\not=2$ and they presented an application to the hyperbolic space $\mathbb H^n$ with $\rho$ being the distance function from the origin point $0$. We refer the reader to \cite{DAD2014} and the references therein for more results in this direction. 

The sharp Hardy inequality on Cartan--Hadamard manifolds $(M,g)$ was recently obtained by Yang, Su and Kong \cite{YSK}
\begin{equation}\label{eq:HardyYSK}
\int_M \frac{|f|^p}{\rho^{p+\beta}}dV_g \leq \lt(\frac{p}{n-p-\beta}\rt)^p \int_M \frac{|\na_g f|_g^p}{\rho^\beta} dV_g,\quad f\in C_0^\infty(M),
\end{equation}
where $n\geq 3$, $p\in (1,n)$, $\beta < n-p$ and $\rho(x) =d(x,P)$ for a fix point $P$ in $M$ (we refer the reader to Section 2 for more about on Cartan--Hadamard manifolds). Furthermore, the constant $p^p/(n-p-\beta)^p$ appeared in \eqref{eq:HardyYSK} is sharp. Our first aim in this paper is to establish an stronger version of \eqref{eq:HardyYSK} as follows: let $\pa_\rho$ denote the derivation along the geodesic curves starting from $P$, then our first result is as follows
\[
\int_M \frac{|f|^p}{\rho^{p+\beta}} dV_g \leq \lt(\frac p{n-p-\beta}\rt)^p \int_M \frac{|\pa_\rho f|^p}{\rho^\beta} dV_g, \qquad f\in C_0^\infty(M)
\]
with the constant $p^p/(n-p-\beta)^p$ being sharp. Obviously, our result is stronger than \eqref{eq:HardyYSK} because of Gauss's lemma $|\pa_\rho f| \leq |\na_g f|_g$. This inequality also extends a result of Ioku et al. \cite{IIO2017} (see also \cite{MOW2017}) on $\R^n$ to the Cartan--Hadamard manifolds. We also study the Hardy inequality in the critical case $p = n-\beta$ and obtain the following inequality
\[
\lt(\frac{p-1}p\rt)^p \int_{B_1(P)} \frac{|f|^p}{\rho^n \lt(\ln \frac1{\rho}\rt)^p} dV_g \leq \int_{B_1(P)} \frac{|\pa_\rho f|^p}{\rho^{n-p}} dV_g,\quad f\in C_0^\infty(B_1(P))
\]
with $(p-1)^p/p^p$ being the sharp constant and $B_1(P)$ denoting the geodesic unit ball around $P$. The case $p=n$ gives us an extension of the inequality \eqref{eq:criticalII} to the Cartan--Hadamard manifolds. In my knowledge, the critical Hardy inequality on Cartan--Hadamard manifolds seems to be new in literature.

Rellich inequalities are the higher order derivative version of Hardy inequality. Let us denote
\begin{equation}\label{eq:constantc}
c(n,2l,\beta,p) = \lt(\prod_{i=0}^{l-1} \frac{p^2}{(n-2p-\beta-2ip)(n(p-1)+\beta +2ip)}\rt)^p
\end{equation}
for $l\geq 1$, $p\in (1, n/(2l))$ and $-n(p-1) < \beta < n-2lp$. It was proved in \cite{DH1998} that
\begin{equation}\label{eq:RellichRneven}
\int_{\R^n} \frac{|f|^p}{|x|^{\beta+ 2lp}} dx \leq c(n,2l,\beta,p) \int_{\R^n} \frac{|\De^l f|^p}{|x|^{\beta}} dx,\quad f\in C_0^\infty(\R^n)
\end{equation}
if $l\geq 1$, $p\in (1, n/(2l))$ and $-2(p-1) < \beta < n-2lp$, and
\begin{equation}\label{eq:RellichRnodd}
\int_{\R^n} \frac{|f|^p}{|x|^{\beta+ (2l+1)p}} dx \leq \frac{p^p}{(n-p-\beta)^p}c(n,2l,p+\beta,p) \int_{\R^n} \frac{|\na \De^l f|^p}{|x|^{\beta}} dx,\quad f\in C_0^\infty(\R^n)
\end{equation}
if $l\geq 1$, $p\in (1, n/(2l+1))$ and $2-3p < \beta < n-(2l+1)p$. Again, these inequalities \eqref{eq:RellichRneven} and \eqref{eq:RellichRnodd} are shown to be sharp in \cite{DH1998}. To prove \eqref{eq:RellichRneven} and \eqref{eq:RellichRnodd}, Davies and Hinz firstly proved \eqref{eq:RellichRneven} for $l=1$ and then iterating this inequality together with the weighted Hardy inequality. The iterative approach is now a standard method to study the inequality of this type. The other approaches are given by Mitidieri \cite{Mi} based on the divergence theorem and the Rellich--Pohozaev type identities \cite{Mi1}. For the improvements of \eqref{eq:RellichRneven} and \eqref{eq:RellichRnodd} on bounded domains and $p=2$, the reader may consult the paper \cite{TZ2007}.

The Rellich inequality was generalized to the Riemannian manifolds in the work of Kombe and \"Ozaydin \cite{KO2009} for $p=2$
\begin{equation}\label{eq:KOR}
\int_M \frac{|\De_g u|^2}{\rho^\al} dV_g \geq \frac{(C-\al-3)^2(C+\al+1)^2}{16} \int_M \frac{|u|^2}{\rho^{\al+4}} dV_g, \quad u\in C_0^\infty(M)
\end{equation}
where $\rho$ is a nonnegative weight function such that $|\na_g \rho|_g =1$, $\De_g \rho \geq C/\rho$ with $\al > -2$ and $\al < C-3$. In the case of Cartan--Hadamard manifold $(M,g)$, the inequality \eqref{eq:KOR}  was  extended by Yang, Su and Kong \cite{YSK} to  $p\not=2$ and to higher order derivatives with $\rho(x) = d(x,P)$ the geodesic distance from $x$ to a fix point $P\in M$. The results of Yang, Su and Kong \cite{YSK} read as follows:
\begin{equation}\label{eq:YSKeven}
\int_{M} \frac{|f|^p}{\rho^{\beta + 2lp}} dV_g \leq c(n,2l,\beta,p) \int_{M} \frac{|\De_g^l f|^p}{\rho^{\beta}} dV_g,\quad f\in C_0^\infty(M),
\end{equation}
with $l\geq 1$, $p\in (1,n/(2l))$ and $-2(p-1) < \beta < n-2lp$, and 
\begin{equation}\label{eq:YSKodd}
\int_{M} \frac{|f|^p}{\rho^{\beta + (2l+1)p}} dV_g \leq \frac{p^p}{(n-p-\beta)^p}c(n,2l,\beta+p,p) \int_{M} \frac{|\na_g \De_g^l f|_g^p}{\rho^{\beta}} dV_g,\quad f\in C_0^\infty(M),
\end{equation}
with $l\geq 1$, $p\in (1,n/(2l+1))$ and $2-3p < \beta < n-(2l+1)p$. Yang, Su and Kong also proved that these inequalities are sharp on $M$. Our next results show that the inequalities \eqref{eq:YSKeven} and \eqref{eq:YSKodd} also true with the same best constants if we replace the Laplace--Beltrami operator by the radial Laplace $\De_{g,\rho}$ taked along the geodesic curve starting from $P$ (see Section 2 below for the precise definition of $\De_{g,\rho}$). The following sharp inequalities will be proved in this paper
\begin{equation}\label{eq:RMeven}
\int_{M} \frac{|f|^p}{\rho^{2lp + \beta}} dV_g \leq c(n,2l,\beta,p) \int_M \frac{|\De_{g,\rho}^lf|^p }{\rho^\beta} dV_g,\quad f\in C_0^\infty(M),
\end{equation}
if $l\geq 1$, $p\in (1,n/(2l))$ and $n(1-p) < \beta < n-2lp$, and
\begin{equation}\label{eq:RModd}
\int_{M} \frac{|f|^p}{\rho^{(2l+1)p + \beta}} dV_g \leq \frac{p^p}{(n-p-\beta)^p}c(n,2l,p+\beta,p)\int_M \frac{|\pa_\rho \De_{g,\rho}^l f|^p}{\rho^\beta} dV_g,\quad f\in C_0^\infty(M),
\end{equation}
if $l\geq 1$, $p\in (1, n/(2l+1))$ and $n-(n+1)p < \beta < n-(2l+1)p$. It is emphasized here that the domain of $\beta$ is extended in your inequalities comparing with \eqref{eq:YSKeven} and \eqref{eq:YSKodd}. In the critical $\beta = n -2lp$ or $n=(2l+1)p$, we will prove the following critical Rellich type inequalities which we believe to be new in Cartan--Hadamard manifolds
\begin{equation}\label{eq:CRMeven}
\int_{B_1(P)} \frac{|f|^p}{\rho(x)^{n} \lt(\ln \frac1{\rho(x)}\rt)^p} dV_g \leq \lt( p'\frac{2^{1-l}}{(l-1)!}\prod_{i=0}^{l-1} \frac1{n -2i-2}\rt)^p \int_{B_1(P)} \frac{|\De_{g,\rho}^lf|^p }{\rho(x)^{n-2lp}} dV_g,
\end{equation}
with $p' =p/(p-1)$, and
\begin{equation}\label{eq:CRModd}
\int_{B_1(P)} \frac{|f|^p}{\rho(x)^{n}\lt(\ln \frac1{\rho(x)}\rt)^p} dV_g \leq \lt(p'\frac1{2^l l!}\prod_{i=0}^{l-1} \frac1{n -2i-2}\rt)^p \int_{B_1(P)} \frac{|\pa_\rho \De_{g,\rho}^l f|^p}{\rho(x)^{n-(2l+1)p}} dV_g,
\end{equation}
for any $f\in C_0^\infty(B_1(P)$. We also show that these inequalities \eqref{eq:CRMeven} and \eqref{eq:CRModd} are sharp. It is worthy to note that  our inequalities \eqref{eq:CRMeven} and \eqref{eq:CRModd} contain the critical Rellich inequalities on the Euclidean space $\R^n$ due to Adimurthi and Santra \cite{AS2009}. In fact, Adimurthi and Santra proved \eqref{eq:CRMeven} and \eqref{eq:CRModd} for radial functions supported in the unit ball centerd at origin in $\R^n$ and for radial function $f$, $\De_{g,\rho} f$ is exactly its Laplace. We should mention here that in the setting of homogeneous groups, the inequalities of the type \eqref{eq:RMeven}, \eqref{eq:RModd}, \eqref{eq:CRMeven} and \eqref{eq:CRModd} was recently established by the author in \cite{Nguyen2017} which extend the result of Ruzhansky and Suragan \cite{RS2017,RS2016,RS2017L2} to the higher order derivatives. 

The last remark is that if the sectional curvature $K_M$ of $(M,g)$ satisfies $K_M \leq -b \leq 0$ along each plane section at each point of $M$, then all our obtained inequalities in this paper will be strengthened. In this situation, we will prove the quantitative versions of Hardy, Rellich, critical Hardy and critical Rellich type inequalities on Cartan--Hadamard manifolds in terms of the upper bounded of $K_M$ (see Sections 3 and 4 below). Especially, in the case of hyperbolic space $\mathbb H^n$ with $K_{\mathbb H^n} \equiv -1$, we show that our obtained Rellich inequality is stronger than the classical one in \cite{KO2009} (see Section 5 below). This is an immediate consequence of the following Machihara--Ozawa--Wadade type result in hyperbolic space (see Theorem \ref{MOWtype})
\[
\int_{\mathbb H^n} |\De_{g,\rho} f|^2 dV_g \leq \int_{\mathbb H^n} |\De_g f|^2 dV_g,\quad f\in C_0^\infty(\mathbb H^n).
\]
In Euclidean space, the previous inequality was prove by Machihara, Ozawa and Wadade \cite{MOW2017}. We also prove several improved Hardy and Rellich type inequalities in hyperbolic space $\mathbb H^n$ (see Theorems \ref{Hyperbolic}, \ref{ImprovedR} and \ref{eq:higherorder} below).

The rest of this paper is organized as follows. In Section 2, we recall some basic notions and properties of Riemannian manifolds (especially, of Cartan--Hadamard manifolds) which we frequently use in this paper. In Sectiona 3, we prove Hardy type inequalities on Cartan--Hadamard manifolds such as weighted Hardy inequality, critical Hardy inequality. We also establish the quantitative improvements for these Hardy type inequalities. Section 4, we prove Rellich type inequalities (both in the critical case and subcritical case) and their quantitative version on Cartan--Hadamard manifolds. In the last section, we establish some improvements of Hardy and Rellich inequalities in hyperbolic spaces.


\section{Preliminaries}
In this section, we list some useful properties of Riemannian manifolds which will be used in this paper. We refer the reader to the book of Helgason \cite{Helgason} and the book of Gallot, Hulin, and Lafontaine \cite{GHL} for standard references. Let $(M,g)$ be an $n-$dimensional complete Riemannian manifold. In a local coordinate system $\{x^i\}_{i=1}^n$, we can write
\[
g = \sum_{i,j=1}^n g_{ij} dx^i dx^.
\]
In such a local coordinate system, the Laplace-Beltrami operator $\De_g$ with respect to the metric $g$ is of the form
\[
\De_g = \sum_{i,j=1}^n \frac1{\sqrt{|g|}} \frac{\partial}{\partial x_i}\lt(\sqrt{|g|} g^{ij} \frac{\partial}{\partial x^j}\rt),
\]
where $|g| =\text{\rm det}(g_{ij})$ and $(g^{ij}) = (g_{ij})^{-1}$. Let us denote by $\na_g$ the corresponding gradient. Then
\[
\la \na_g f,\na_g h\ra = \sum_{i,j=1}^n g^{ij} \frac{\pa f}{\pa x^i} \frac{\pa h}{\pa x^j}.
\]
For simplicity, we shall use the notation $|\na_g f|_g=\sqrt{\la \na_g f,\na_g f\ra}$.


Let $K_M$ be the sectional curvature on $(M,g)$. A Riemannian manifold $(M,g)$ is called Cartan--Hadamard manifolds if it is complete, simply connected and with non-positive sectional curvature, i.e., $K_M \leq 0$ along each plane section at each point of $M$. If $(M,g)$ is Cartan--Hadamard manifold, then for each point $P\in M$, $M$ contains no points conjugate to $P$. Moreover, the exponential map $\text{\rm Exp}_P: T_PM \to M$ is a diffeomorphism, where $T_PM$ is the tangent space to $M$ at $P$ (see, e.g., \cite[Chapter ${\rm I}$]{Helgason})

Let $(M,g)$ be a Cartan--Hadamard manifold. Fix a point $P\in M$ and denote by $\rho(x) =d(x,P)$ for all $x \in M$, where $d$ denotes the geodesic distance on $M$. Note that $\rho(x)$ is smooth on $M\setminus\{P\}$ and satisfies 
\[
|\na_g \rho(x)|_g  =1,\qquad x \in M \setminus\{P\}.
\]
Moreover, since $\text{\rm Exp}_P$ is a diffeomorphism, then the function
\[
\rho(x)^2 = \|\text{\rm Exp}_P^{-1}(x)\|^2 \in C^\infty(M).
\]
The radial derivation $\pa_\rho = \frac{\pa}{\pa \rho}$ along geodesic curve starting from $P$ is defined for any function $f$ on $M$ by 
\[
\pa_\rho f(x) = \frac{d (f\circ \text{\rm Exp}_P)}{dr}(\text{\rm Exp}_P^{-1}(x)),
\]
here $\frac{d}{dr}$ denotes the radial derivation on $T_PM$, i.e., 
\[
\frac{d}{dr} F(u) = \lt\la \frac{u}{|u|}, \na F(u)\rt\ra,\qquad u\in T_PM \setminus\{0\}.
\]

For any $\de >0$, denote by $B_\de(P) = \{x\in M\, :\, \rho(x) < \de\}$ the geodesic ball in $M$ with center at $P$ and radius $\de$. We introduce the density function $J(u,t)$ of the volume form in normal coordinates as follows (see, e.g., \cite[pp. $166-167$]{GHL}). Choose an orthonormal basis $\{u, e_2,\ldots,e_n\}$ on $T_PM$ and let $c(t) =\text{\rm Exp}_P(tu)$ be a geodesic curve. The Jacobian fields $\{Y_i(t)\}_{i=2}^n$ satisfy $Y_i(0) =0$, $Y_i'(0) =e_i$, so that the density function can be given by
\[
J(u,t) = t^{1-n} \sqrt{\text{\rm det}(\la Y_i(t), Y_j(t)\ra)},\quad t >0.
\]
We note that $J(u,t)$ does not depend on $\{e_2,\ldots,e_n\}$ and $J(u,t)\in C^\infty(T_PM \setminus\{0\})$ by the definition of $J(u,t)$. Moreover, if we set $J(u,0)\equiv 1$ then $J(u,t) \in C(T_PM)$ and has the following asymptotic expansion
\begin{equation}\label{eq:density}
J(u,t) = 1 + O(t^2)
\end{equation}
as $t\to 0^+$ since $Y_i(t)$ has the asymptotic expansion (see, e.g., \cite[p. $169$]{GHL})
\[
Y_i(t) = t e_i -\frac{t^3}6 R(c'(t), e_i) c'(t) + o(t^3),
\]
as $t\to 0^+$, where $R(\cdot, \cdot)$ is the curvature tensor on $M$.

From the definition of $J(u,t)$, we have the following polar coordinate on $M$
\begin{equation}\label{eq:polar}
\int_M f dV_g = \int_0^\infty \int_{S^{n-1}} f(\text{\rm Exp}_P(tu)) J(u,t) t^{n-1} dt du, 
\end{equation}
where $du$ denotes the canonical measure of the unit sphere of $T_PM$. Moreover, the Laplacian of the distance function $\rho(x)$ has the following expansion via the function $J(u,t)$ (see, e.g., \cite[$4.B.2$]{GHL})
\begin{equation}\label{eq:Lapofdist}
\De_g \rho = \frac{n-1}\rho + \frac{J'(u,\rho)}{J(u,\rho)},\quad \rho >0,
\end{equation}
where $J'(u,\rho) = \frac{\pa J(u,\rho)}{\pa \rho}$. Therefore, for any radial function $f(\rho)$ on $M$, we have
\begin{equation}\label{eq:Lapofradial}
\De_g f(\rho) = f''(\rho) + \lt(\frac{n-1}\rho + \frac{J'(u,\rho)}{J(u,\rho)}\rt) f'(\rho).
\end{equation}
Let us introduce the radial Laplacian of a function $f$ by
\begin{equation}\label{eq:radialLap}
\De_{g,\rho f(x)} = \pa_\rho^2 f(x) + \lt(\frac{n-1}\rho + \frac{J'(u,\rho)}{J(u,\rho)}\rt) \pa_\rho f(x),\quad x = \text{\rm Exp}_P(tu).
\end{equation}
Note that if $K_M$ is constant then $J(u,t)$ depends only on $t$. We denote by $J_b(t)$ the corresponding density function if $K_M\equiv -b$ for some $b\geq 0$, i.e., 
\[
J_b(t) = \begin{cases}
1 &\mbox{if $b =0$}\\
\lt(\frac{\sinh(\sqrt{b}t)}{\sqrt{b} t}\rt)^{n-1} &\mbox{if $b>0$}.
\end{cases}
\]
For $b \geq 0$, we consider the function $\text{\bf ct}_b :(0,\infty) \to \R$ defined by
\[
\text{\bf ct}_b(t) = 
\begin{cases}
\frac1 t&\mbox{if $b=0$}\\
\sqrt{b} \coth(\sqrt{b} t)&\mbox{if $b>0$},
\end{cases}
\]
and the function $\text{\bf D}_b :[0,\infty) \to \R$ defined by
\[
\text{\bf D}_b(t) = 
\begin{cases}
0 &\mbox{if $t =0$}\\
t \text{\bf ct}_b(t) -1 &\mbox{if $t>0$}.
\end{cases}
\]
Clearly, we have $\text{\bf D}_b \geq 0$. If the section curvature $K_M$ on $M$ satisfies $K_M \leq -b$ then the Bishop--Gunther comparison theorem (see, e.g., \cite[p. $172$]{GHL} for its proof) says that
\begin{equation}\label{eq:Bishop}
\frac{J'(u,t)}{J(u,t)} \geq \frac{J_b'(t)}{J_b(t)} = \frac{n-1}t \bD_b(t),\qquad t>0.
\end{equation}




\section{Hardy type inequalities}
This section is devoted to proved the Hardy type inequalities on Cartan--Hadamard manifolds $(M,g)$. By \eqref{eq:Bishop}, the function $t \mapsto \rho(u,t)$ is non-decreasing monotone on $(0,\infty)$ for any $u\in S^{n-1}$. We first have the following weighted Hardy inequalities 

\begin{theorem}\label{Hardy}
Let $(M,g)$ be a $n-$dimensional Cartan--Hadamard manifolds. Suppose that $n \geq 2$, $p\in (1,n)$ and $\beta < n-p$. There holds for any $f \in C_0^\infty(M)$
\begin{equation}\label{eq:subcriticalH}
\int_M \frac{|f(x)|^p}{\rho(x)^{p+\beta}} dV_g \leq \lt(\frac p{n-p-\beta}\rt)^p \int_M \frac{|\pa_\rho f(x)|^p}{\rho(x)^\beta} dV_g.
\end{equation}
Furthermore, the constant $(\frac p{n-p-\beta})^p$ is the best constant in \eqref{eq:subcriticalH}.
\end{theorem}
By Gauss's lemma (see, e.g., \cite{GHL}), we have 
\[
|\pa_\rho f| \leq |\na_g f|_g,\qquad f\in C^1(M).
\]
Hence, the inequality \eqref{eq:subcriticalH} is stronger than the weighted Hardy inequality obtained in \cite{KO2009} (see also \cite{YSK})
\[
\int_M \frac{|f(x)|^p}{\rho(x)^{p+\beta}} dV_g \leq \lt(\frac p{n-p-\beta}\rt)^p \int_M \frac{|\na_g f(x)|_g^p}{\rho(x)^\beta} dV_g,\quad f\in C_0^\infty(M)
\]
for any $\beta < n-p$.

The inequality \eqref{eq:subcriticalH} was proved by Ioku et al. \cite{IIO2017} and by Machihara et al. \cite{MOW2016} when $M =\R^n$. Furthermore, they proved inequality \eqref{eq:subcriticalH} with an extra sharp remainder term. In the proof of the inequality \eqref{eq:subcriticalH} below, we also get the inequality \eqref{eq:subcriticalH} with extra remainder terms: one comes from the density function $J(u,t)$ and the other is as in the one of Ioku, Ishiwata and Ozawa \cite{IIO2017} and Machihara, Ozawa and Wadade \cite{MOW2016}. In order to prove Theorem \ref{Hardy}, let us introduce the quantity $R_p(\xi,\eta)$ for $p >1$ and $\xi, \eta \in T_PM$ by
\begin{equation}\label{eq:Remainder}
R_p(\xi, \eta) = \frac1p |\eta|^p + \frac{p-1}p |\xi|^p -|\xi|^{p-2} \la \xi, \eta\ra.
\end{equation}
This quantity closely relates to the remainder term in \eqref{eq:subcriticalH} (see, also, \cite{IIO2017,MOW2016}). By the convexity of $\xi \to |\xi|^p$ we see that $R_p(\xi,\eta) \geq 0$ with equality if and only if $\xi = \eta$. Furthermore, we can see that
\[
R_p(\xi, \eta) = (p-1)\int_0^1 |t\xi + (1-t)\eta|^{p-2} t dt |\xi -\eta|^2.
\]
\begin{proof}[Proof of Theorem \ref{Hardy}]
Let $f \in C_0^\infty(M)$, then the function
\[
F(y) = f(\text{\rm Exp}_P(y)) \in C_0^\infty(T_PM).
\]
Using the polar coordinate \eqref{eq:polar} and the assumption $\beta < n-p$, we get
\begin{align}\label{eq:applypolar1}
\int_M \frac{|f(x)|^p}{\rho(x)^{p+ \beta}} dV_g&=\int_{S^{n-1}} \int_0^\infty |F(\rho u)|^p J(u,\rho) \rho^{n-p-\beta -1} d\rho du\notag\\
&=\frac{1}{n-p-\beta} \int_{S^{n-1}} \int_0^\infty |F(\rho u)|^p J(u,\rho) d\rho^{n-p-\beta} du\notag\\
&=-\frac p{n-p-\beta} \int_{S^{n-1}}\int_0^\infty |F|^{p-2} F \pa_\rho F J(u,\rho) \rho^{n-p-\beta} d\rho du\notag\\
&\quad -\frac1{n-p-\beta} \int_{S^{n-1}} \int_0^\infty |F|^p J'(u,\rho) \rho^{n-p-\beta} d\rho du\notag\\
&=-\frac p{n-p-\beta} \int_M\frac{|f|^{p-2} f}{\rho(x)^{\frac{p-1}p(p+\beta)}} \frac{\pa_\rho f}{\rho(x)^{\frac\beta p}} dV_g\notag\\
&\quad -\frac1{n-p-\beta} \int_{S^{n-1}} \int_0^\infty |F|^p J'(u,\rho) \rho^{n-p-\beta} d\rho du.
\end{align}
Using the definition of $R_p$, we then have
\begin{align}\label{eq:identity1}
-\frac p{n-p-\beta} \int_M & \frac{|f|^{p-2} f}{\rho(x)^{\frac{p-1}p(p+\beta)}} \frac{\pa_\rho f}{\rho(x)^{\frac\beta p}} dV_g\notag\\
 &= \frac{p-1}p \int_M \frac{|f|^p}{\rho(x)^{p+ \beta}} dV_g + \frac1p\lt(\frac p{n-p-\beta}\rt)^p \int_M \frac{|\pa_\rho f|^p}{\rho(x)^\beta} dV_g\notag\\
&\quad - \int_M R_p\lt(\frac{f}{\rho(x)^{1+ \frac\beta p}}, -\frac p{n-p-\beta} \frac{\pa_\rho f}{\rho(x)^{\frac\beta p}}\rt) dV_g.
\end{align}
Combining \eqref{eq:applypolar1} and \eqref{eq:identity1} together, we get
\begin{align}\label{eq:identity}
\int_M \frac{|f(x)|^p}{\rho(x)^{p+ \beta}} dV_g& = \lt(\frac p{n-p-\beta}\rt)^p \int_M \frac{|\pa_\rho f|^p}{\rho(x)^\beta} dV_g - p\int_M \frac{R_p\lt(\frac{f}{\rho(x)}, -\frac p{n-p-\beta} \pa_\rho f\rt)}{\rho(x)^\beta} dV_g\notag\\
&\quad -\frac{p}{n-p-\beta} \int_M \frac{|f|^p}{\rho(x)^{p+\beta}} \frac{J'(u_x,\rho(x))\rho(x)}{J(u_x,\rho(x))} dV_g,
\end{align}
here for $x\in M\setminus\{P\}$, we denote by $u_x$ the unique unit vector in $T_PM$ such that $x =\text{\rm Exp}_P(\rho(x) u_x)$.

The inequality \eqref{eq:subcriticalH} is now an immediate consequence of \eqref{eq:identity} by dropping the nonnegative remainder terms. It remains to check the sharpness of constant. To do this, we approximate the function $\rho^{-(n-p-\beta)/p}$ as follows. Let $\phi \in C_0^\infty(\R)$ such that $0\leq \phi \leq 1$, $\phi(t) =1$ if $|t| \leq 1$ and $\phi(t) =0$ if $|t|\geq 2$. For $\ep >0$, define
\[
f_\ep(x) = \phi(\rho(x)) (1 -\phi(\ep^{-1} \rho(x))) \rho(x)^{-\frac{n-p-\beta}p}.
\]
A straightforward computation shows that
\begin{align*}
\int_M \frac{f_\ep(x)^p}{\rho(x)^{p+\beta}} dV_g &= \int_{S^{n-1}} \int_0^\infty \phi(t)^p (1-\phi(\ep^{-1}t))^p J(u,t) t^{-1} dt\\
&\geq \int_{S^{n-1}}\int_{2\ep}^{1}t^{-1} dt du \geq -|S^{n-1}| \ln(2\ep),
\end{align*}
here we use the increasing monotonicity of $J(u,t)$. Consequently
\[
\lim_{\ep\to 0^+}\int_M \frac{f_\ep(x)^p}{\rho(x)^{p+\beta}} dV_g = \infty.
\]
In the other hand, we have
\begin{align*}
\pa_\rho f_\ep &= \phi'(\rho(x))(1 -\phi(\ep^{-1} \rho(x))) \rho(x)^{-\frac{n-p-\beta}p} -\ep^{-1} \phi'(\ep^{-1}\rho(x)) \phi(\rho(x))\rho(x)^{-\frac{n-p-\beta}p}\\
&\quad -\frac{n-p-\beta}p \phi(\rho(x)) (1 -\phi(\ep^{-1} \rho(x))) \rho(x)^{-\frac{n-\beta}p}.
\end{align*}
Easy computations shows that
\[
\int_{M}\frac{|\phi'(\rho(x))(1 -\phi(\ep^{-1} \rho(x))) \rho(x)^{-\frac{n-p-\beta}p}|^p}{\rho^{\beta}} dV_g = O(1),
\]
\[
\int_{M} \frac{|\ep^{-1} \phi'(\ep^{-1}\rho(x)) \phi(\rho(x))\rho(x)^{-\frac{n-p-\beta}p}|^p}{\rho^\beta} dV_g = O(1),
\]
and 
\[
\int_M \frac{|\phi(\rho(x)) (1 -\phi(\ep^{-1} \rho(x))) \rho(x)^{-\frac{n-\beta}p}|^p}{\rho(x)^{\beta}} dx = \int_M \frac{f_\ep(x)^p}{\rho(x)^{p+\beta}} dV_g.
\]
These estimates prove the sharpness of \eqref{eq:subcriticalH}.
\end{proof}

It is worthy to note that the proof above of Theorem \ref{Hardy} also gives a quantitative Hardy type inequality on Cartan--Hadamard manifolds.
\begin{theorem}\label{quanH}
Suppose the assumptions in statement of Theorem \ref{Hardy} and suppose that $K_M \leq -b \leq 0$. Then for any $f\in C_0^\infty(M)$, we have
\begin{equation}\label{eq:quantitativeH}
\int_M \frac{|\pa_\rho f|^p}{\rho^\beta} dV_g\geq \lt(\frac{n-p-\beta}p\rt)^p \int_{M} \frac{|f|^p}{\rho^{p+ \beta}}\lt(1+ \frac{(n-1)p}{n-p-\beta}\bD_b(\rho)\rt) dV_g.
\end{equation}
Consequently, the following quantitative Hardy inequalities hold
\begin{equation}\label{eq:quantitativeHa}
\int_M \frac{|\na_g f|_g^p}{\rho^\beta} dV_g\geq \lt(\frac{n-p-\beta}p\rt)^p \int_{M} \frac{|f|^p}{\rho^{p+ \beta}}\lt(1+ \frac{(n-1)p}{n-p-\beta}\bD_b(\rho)\rt) dV_g,
\end{equation}
\begin{align}\label{eq:quantitativeHb1}
\int_M \frac{|\pa_\rho f|^p}{\rho^\beta} dV_g&\geq \lt(\frac{n-p-\beta}p\rt)^p \int_{M} \frac{|f|^p}{\rho^{p+ \beta}}dV_g\notag\\
&\quad + 3b(n-1) \lt(\frac{n-p-\beta}p\rt)^{p-1}\int_{M} \frac{|f|^p}{\rho^{p+ \beta-2}(\pi^2 + b\rho^2)} dV_g.
\end{align}
and
\begin{align}\label{eq:quantitativeHb}
\int_M \frac{|\na_g f|_g^p}{\rho^\beta} dV_g&\geq \lt(\frac{n-p-\beta}p\rt)^p \int_{M} \frac{|f|^p}{\rho^{p+ \beta}}dV_g\notag\\
&\quad + 3b(n-1) \lt(\frac{n-p-\beta}p\rt)^{p-1}\int_{M} \frac{|f|^p}{\rho^{p+ \beta-2}(\pi^2 + b\rho^2)} dV_g.
\end{align}
\end{theorem}
The inequalities \eqref{eq:quantitativeHa} and \eqref{eq:quantitativeHb} was proved by Krist\'aly for $p=2$ and $\beta =0$ in \cite[Theorems $4.1$ and $4.2$]{K2} on Cartan--Hadamard manifolds. \eqref{eq:quantitativeHa} was extended to Finsler--Hadamard manifolds again for $p=2$ by Krist\'aly and Repov$\check{\text{\rm s}}$ in \cite[Lemma $3.1$]{K1} (see also \cite{FKV} for a particular form). Obviously, \eqref{eq:quantitativeHa} and \eqref{eq:quantitativeHb} provide the improvements of the Hardy inequality on Cartan--Hadamard manifolds due to Yang, Su and Kong \cite[Theorem $3.1$]{YSK} if $b >0$.
\begin{proof}
The inequality \eqref{eq:quantitativeH} is followed from \eqref{eq:identity} and \eqref{eq:Bishop}. The inequality \eqref{eq:quantitativeHa} is consequence of \eqref{eq:quantitativeH} and Gauss lemma. The inequality \eqref{eq:quantitativeHb1} is derived from \eqref{eq:quantitativeH} and the simple fact (see the proof of Theorem $1.4$ in \cite{K2})
\[
t \coth (t) -1 \geq \frac{3t^2}{\pi^2 + t^2}, \quad t>0.
\]
The inequality \eqref{eq:quantitativeHb} is followed from \eqref{eq:quantitativeHb1} and Gauss lemma (or from \eqref{eq:quantitativeHa} and the simple fact above).
\end{proof}

We next consider the critical case $\beta =n-p$. In this case, the inequality \eqref{eq:subcriticalH} fails for any constant. In order to establish the inequalities in this case, we need add an extra logarithmic term. The next result of this section reads as follows.  

\begin{theorem}\label{CH}
Let $n\geq 2$ and let $M$ be an $n$-dimensional Cartan--Hadamard manifold. Then, for any $p\in (1,\infty)$, there holds
\begin{equation}\label{eq:criticalH}
\lt(\frac{p-1}p\rt)^p \int_{B_1(P)} \frac{|f|^p}{\rho(x)^n \lt(\ln \frac1{\rho(x)}\rt)^p} dV_g \leq \int_{B_1(P)} \frac{|\pa_\rho f|^p}{\rho(x)^{n-p}} dV_g,
\end{equation}
for any function $f \in C_0^\infty(B_1(P))$. Furthermore, the constant $(\frac{p-1}p)^p$ is the best constant in \eqref{eq:criticalH}.
\end{theorem}
\begin{proof}
Suppose that $f \in C_0^\infty(B_1(P))$ then
\[
F(y) = f(\text{\rm Exp}_P(y)) \in C_0^\infty(B_1),
\]
where $B_1$ denotes the unit ball in $T_PM$ with center at origin. Applying polar coordinate \eqref{eq:polar}, we have
\begin{align}\label{eq:applypolar2}
\int_{B_1(P)} \frac{|f|^p}{\rho(x)^n \lt(\ln \frac1{\rho(x)}\rt)^p} dV_g&= \int_{S^{n-1}} \int_0^1 |F(\rho u)|^p J(u,\rho) \rho^{-1} \lt(-\ln \rho\rt)^{-p} d\rho du\notag\\
&=\frac1{p-1} \int_{S^{n-1}}\int_0^1 |F(\rho u)|^p J(u,\rho) d\lt(-\ln \rho\rt)^{1-p} du \notag\\
&=-\frac{p}{p-1}\int_{S^{n-1}} \int_0^1 \frac{|F|^{p-2} F \pa_\rho F J(u,\rho)}{ \lt(-\ln \rho\rt)^{p-1}} d\rho du\notag\\
&\quad -\frac1{p-1}\int_{S^{n-1}} \int_0^1 |F(\rho u)|^p J'(u,\rho)\lt(-\ln \rho\rt)^{1-p} d\rho du\notag\\
&=-\frac{p}{p-1} \int_{B_1(P)} \frac{|f|^{p-2}f}{\rho(x)^{n\frac{p-1}p}\lt(-\ln \rho\rt)^{p-1}} \frac{\pa_\rho f}{\rho(x)^{\frac np-1}} dV_g\notag\\
&\quad -\frac1{p-1}\int_{B_1(P)} \frac{|f|^p}{\rho^n \lt(\ln \frac1{\rho}\rt)^p} \frac{J'(u_x,\rho)}{J(u_x,\rho)} \rho \ln \frac1{\rho} dV_g.
\end{align}
From the definition of $R_p$, we get
\begin{align}\label{eq:identity2}
-\frac{p}{p-1} &\int_{B_1(P)} \frac{|f|^{p-2}f}{\rho(x)^{n\frac{p-1}p}\lt(-\ln \rho\rt)^{p-1}} \frac{\pa_\rho f}{\rho(x)^{\frac np-1}} dV_g\notag\\
&=\frac{p-1}p\int_{B_1(P)} \frac{|f|^p}{\rho(x)^n \lt(\ln \frac1{\rho(x)}\rt)^p} dV_g + \frac1p \lt(\frac p{p-1}\rt)^p\int_{B_1(P)} \frac{|\pa_\rho f|^p}{\rho(x)^{n-p}} dV_g\notag\\
&\quad -\int_{B_1(P)}R_p\lt(\frac{f}{\rho(x)^{\frac np} (-\ln \rho(x))},-\frac  p{p-1} \frac{\pa_\rho f}{\rho(x)^{\frac np -1}} \rt)dV_g.
\end{align}
Combining \eqref{eq:applypolar2} and \eqref{eq:identity2} together implies
\begin{align}\label{eq:identitycritical}
\int_{B_1(P)} \frac{|f|^p}{\rho^n \lt(\ln \frac1{\rho}\rt)^p} dV_g &=\lt(\frac p{p-1}\rt)^p\int_{B_1(P)} \frac{|\pa_\rho f|^p}{\rho^{n-p}} dV_g -p \int_{B_1(P)} \frac{R_p\lt(\frac{f}{\rho \ln \frac1{\rho}},-\frac  p{p-1} \pa_\rho f \rt)}{\rho^{n-p}} dV_g\notag\\
&\quad -\frac p{p-1}\int_{B_1(P)} \frac{|f|^p}{\rho^n \lt(\ln \frac1{\rho}\rt)^p} \frac{J'(u_x,\rho)}{J(u_x,\rho)} \rho \ln \frac1{\rho} dV_g.
\end{align}
The inequality \eqref{eq:criticalH} is now followed from \eqref{eq:identitycritical} by dropping the nonnegative remainder terms. It remains to check the sharpness of \eqref{eq:criticalH}. Let $\vphi$ be a cut-off function in $(-1,1)$, i.e., $\vphi \in C_0^\infty((-1,1))$ such that $0\leq \vphi \leq 1$, $\vphi(t) =1$ if $|t| \leq 1/2$. For $\de >0$ small enough, define 
\[
f_\de(x) = \lt(\ln \frac1{\rho(x)}\rt)^{\frac{p-1}p -\de} \vphi(\rho(x)).
\]
Firstly, it follows from \eqref{eq:polar} and the increasing monotonicity of $J(u,\rho)$ that
\begin{align*}
\int_{B_1(P)} \frac{|f_\de|^p}{\rho^n \lt(\ln \frac1{\rho}\rt)^p} dV_g&= \int_{S^{n-1}}\int_0^1\lt(\ln \frac1{\rho(x)}\rt)^{-1-p\de} \frac{\vphi(\rho)^p}{\rho} J(u,\rho) d\rho du\\
&\geq |S^{n-1}|\int_0^{1/2} \lt(\ln \frac1{\rho(x)}\rt)^{-1-p\de} \frac1{\rho} d\rho\\
&=|S^{n-1}| \frac1{p\de} \lt(\ln 2\rt)^{-p\de}.
\end{align*}
We thus have
\[
\lim_{\de \to 0^+}\int_{B_1(P)} \frac{|f_\de|^p}{\rho^n \lt(\ln \frac1{\rho}\rt)^p} dV_g = \infty.
\]
In the other hand, the straightforward computations show that
\[
\pa_\rho f_\de = -\frac{p-1-p\de}p\lt(\ln \frac1{\rho}\rt)^{-\frac{1+p\de}p} \vphi(\rho) + \lt(\ln \frac1{\rho}\rt)^{\frac{p -1-p\de}p} \vphi'(\rho),
\]
\[
\int_{B_1(P)} \frac{|\frac{p-1-p\de}p\lt(\ln \frac1{\rho}\rt)^{-\frac{1+p\de}p} \vphi(\rho)|^p}{\rho^{n-p}} dV_g = \lt(\frac{p-1-p\de}p\rt)^p \int_{B_1(P)} \frac{|f_\de|^p}{\rho^n \lt(\ln \frac1{\rho}\rt)^p} dV_g,
\]
and
\[
\int_{B_1(P)} \frac{\lt|\lt(\ln \frac1{\rho}\rt)^{\frac{p -1-p\de}p} \vphi'(\rho)\rt|^p}{\rho^{n-p}} dV_g = O(1).
\]
Consequently, we obtain
\[
\lim_{\de\to 0^+} \frac{\int_{B_1(P)} \frac{|\pa_\rho f_\de|^p}{\rho^{n-p}} dV_g}{\int_{B_1(P)} \frac{|f_\de|^p}{\rho^n \lt(\ln \frac1{\rho}\rt)^p} dV_g} = \lt(\frac{p-1}p\rt)^p.
\]
This proves the sharpness of \eqref{eq:criticalH}.
\end{proof}
Theorem \ref{CH} together with Gauss's lemma yields the following critical Hardy type inequalities for full gradient on $M$,
\begin{equation}\label{eq:criticalHfull}
\lt(\frac{p-1}p\rt)^p \int_{B_1(P)} \frac{|f|^p}{\rho(x)^n \lt(\ln \frac1{\rho(x)}\rt)^p} dV_g \leq \int_{B_1(P)} \frac{|\na_g f|_g^p}{\rho(x)^{n-p}} dV_g,
\end{equation}
for any function $f\in C_0^\infty(B_1(P))$. Using again the test functions in the proof of Theorem \ref{CH}, we see that the constant $(p-1)^p/p^p$ in \eqref{eq:criticalHfull} is sharp.

Similar to the subcritical case, we also obtain from the proof of Theorem \ref{CH} the following quantitative critical Hardy inequalities whose proof is completely similar with the one of Theorem \ref{quanH}.
\begin{theorem}\label{QuanCH}
Suppose the assumptions in statement of Theorem \ref{CH} and suppose $K_M \leq -b \leq 0$. Then the following inequalities hold for any function $f \in C_0^\infty(B_1(P))$
\begin{equation}\label{eq:QuantiCH}
\int_{B_1(P)} \frac{|\pa_\rho f|^p}{\rho^{n-p}} dV_g \geq \lt(\frac{p-1}p\rt)^p \int_{B_1(P)}\frac{|f|^p}{\rho^n (\ln \frac1\rho)^p}\lt(1+ \frac{(n-1)p}{p-1} \bD_b(\rho) \ln \frac1\rho \rt) dV_g,
\end{equation}
\begin{equation}\label{eq:QuantiCHa}
\int_{B_1(P)} \frac{|\na_g f|_g^p}{\rho^{n-p}} dV_g \geq \lt(\frac{p-1}p\rt)^p \int_{B_1(P)}\frac{|f|^p}{\rho^n (\ln \frac1\rho)^p}\lt(1+ \frac{(n-1)p}{p-1} \bD_b(\rho) \ln \frac1\rho \rt) dV_g,
\end{equation}
\begin{align}\label{eq:QuantiCHb1}
\int_{B_1(P)} \frac{|\pa_\rho f|^p}{\rho^{n-p}} dV_g &\geq \lt(\frac{p-1}p\rt)^p \int_{B_1(P)}\frac{|f|^p}{\rho^n (\ln \frac1\rho)^p} dV_g\notag\\
&\quad +3b(n-1) \lt(\frac{p-1}p\rt)^{p-1} \int_{B_1(P)}\frac{|f|^p}{\rho^{n-2} (\ln \frac1\rho)^{p-1}(\pi^2 + b\rho^2)}dV_g,
\end{align}
and
\begin{align}\label{eq:QuantiCHb}
\int_{B_1(P)} \frac{|\na_g f|_g^p}{\rho^{n-p}} dV_g &\geq \lt(\frac{p-1}p\rt)^p \int_{B_1(P)}\frac{|f|^p}{\rho^n (\ln \frac1\rho)^p} dV_g\notag\\
&\quad +3b(n-1) \lt(\frac{p-1}p\rt)^{p-1} \int_{B_1(P)}\frac{|f|^p}{\rho^{n-2} (\ln \frac1\rho)^{p-1}(\pi^2 + b\rho^2)}dV_g.
\end{align}
\end{theorem}

Especially, in the case $p=n$, we obtain from Theorems \ref{CH} and \ref{QuanCH} the following critical Hardy inequalities and their quantitative versions.
\begin{corollary}\label{eq:Critical}
Suppose the assumptions in the statement of Theorem \ref{CH}. Then, there holds for any function $f \in C_0^\infty(B_1(P))$
\begin{equation}\label{eq:criticalHH}
\lt(\frac{n-1}n\rt)^n \int_{B_1(P)} \frac{|f|^n}{\rho(x)^n \lt(\ln \frac1{\rho(x)}\rt)^n} dV_g \leq \int_{B_1(P)} |\pa_\rho f|^n dV_g,
\end{equation}
and
\begin{equation}\label{eq:criticalHHfull}
\lt(\frac{n-1}n\rt)^n \int_{B_1(P)} \frac{|f|^n}{\rho(x)^n \lt(\ln \frac1{\rho(x)}\rt)^n} dV_g \leq \int_{B_1(P)} |\na_g f|_g^n dV_g.
\end{equation}
Furthermore, the constant $(\frac{n-1}n)^n$ is the best constant in \eqref{eq:criticalHH} and \eqref{eq:criticalHHfull}.

Suppose, in addition, $K_M \leq -b \leq 0$, then we have
\begin{align}\label{eq:QuantiCHb1n}
\int_{B_1(P)} |\pa_\rho f|^n dV_g &\geq \lt(\frac{n-1}n\rt)^n \int_{B_1(P)}\frac{|f|^n}{\rho^n (\ln \frac1\rho)^n} dV_g\notag\\
&\quad +3b\frac{(n-1)^n}{n^{n-1}} \int_{B_1(P)}\frac{|f|^n}{\rho^{n-2} (\ln \frac1\rho)^{n-1}(\pi^2 + b\rho^2)}dV_g,
\end{align}
and
\begin{align}\label{eq:QuantiCHn}
\int_{B_1(P)} |\na_g f|_g^n dV_g &\geq \lt(\frac{n-1}n\rt)^n \int_{B_1(P)}\frac{|f|^n}{\rho^n (\ln \frac1\rho)^n} dV_g\notag\\
&\quad +3b \frac{(n-1)^n}{n^{n-1}} \int_{B_1(P)}\frac{|f|^n}{\rho^{n-2} (\ln \frac1\rho)^{n-1}(\pi^2 + b\rho^2)}dV_g
\end{align}
\end{corollary}
In Euclidean space (i.e., $M=\R^n$), the inequality \eqref{eq:criticalHH} was proved by Ioku and Ishiwata \cite{II2015}, and then was extended for any $p>1$ by Ruzhansky and Suragan \cite{RS2016} (i.e, the inequality \eqref{eq:criticalH} in the $\R^n$). More plus, Ruzhansky and Suragan was generalized the inequality \eqref{eq:criticalH} to any homogeneous groups and any homogeneous quasi-norm with the same best constant.

\section{The Rellich type inequalities}
In this section, we study the Rellich type inequalities on Cartan--Hadamard manifolds $(M,g)$. The following result will play an important role in our analysis below.

\begin{lemma}\label{onetwolemma}
Let $(M,g)$ be an $n-$dimensional Cartan--Hadamard manifold. Suppose that $n \geq 2$, $p\in (1,n)$ and $-n(p-1)< \beta < n-p$. There holds for any $f \in C_0^\infty(M)$
\begin{equation}\label{eq:onetwoinequality}
\int_M \frac{|f|^p}{\rho(x)^{p+\beta}} dV_g \leq \lt(\frac p{n(p-1)+\beta}\rt)^p \int_M \frac{|\pa_\rho f + (\frac{n-1}{\rho(x)} + \frac{J'(u_x,\rho(x))}{J(u_x,\rho(x)}) f|^p}{\rho(x)^\beta} dV_g.
\end{equation}
Furthermore, the constant $(\frac p{n(p-1)+\beta})^p$ is the best constant in \eqref{eq:onetwoinequality}.
\end{lemma}
\begin{proof}
Suppose $f \in C_0^\infty(M)$, then
\[
F(y) = f(\text{\rm Exp}_P(y)) \in C^\infty_0(T_PM).
\]
It follows from the polar coordinate \eqref{eq:polar} and integration by parts that
\begin{align}\label{eq:onetwopolar}
\int_M \frac{|f(x)|^p}{\rho(x)^{p+\beta}} dV_g&=\int_{S^{n-1}}\int_0^\infty |F|^p J(u,\rho) \rho^{n-p-\beta-1} d\rho du\notag\\
&=-\frac{p}{n-p-\beta} \int_{S^{n-1}}\int_0^\infty |F|^{p-2}F \pa_\rho F J(u,\rho) \rho^{n-p-\beta} d\rho du\notag\\
&\quad -\frac1{n-p-\beta} \int_{S^{n-1}}\int_0^\infty |F|^p J'(u,\rho) \rho^{n-p-\beta} d\rho du\notag\\
&= -\frac{p}{n-p-\beta} \int_M \frac{|f|^{p-2} f}{\rho(x)^{\frac{p-1}p(p+ \beta)}} \frac{\pa_\rho f}{\rho(x)^{\frac\beta p}}dV_g\notag\\
&\quad -\frac1{n-p-\beta}\int_M \frac{|f|^p}{\rho(x)^{p+\beta}} \frac{J'(u_x,\rho(x))\rho(x)}{J(u_x,\rho(x))} dV_g,
\end{align}
here we use $\beta < n-p$. In the other hand, we have
\begin{align}\label{eq:onetwoadd}
\int_M \frac{|f|^{p-2} f}{\rho^{\frac{p-1}p(p+ \beta)}} \frac{\pa_\rho f}{\rho^{\frac\beta p}}dV_g& = \int_M \frac{|f|^{p-2} f}{\rho(x)^{\frac{p-1}p(p+ \beta)}} \frac{\pa_\rho f + (\frac{n-1}{\rho(x)} + \frac{J'(u_x,\rho(x))}{J(u_x,\rho(x))})f}{\rho(x)^{\frac\beta p}}dV_g\notag\\
&\quad -(n-1)\int_M \frac{|f|^p}{\rho^{p+\beta}} dV_g -\int_M \frac{|f|^p}{\rho(x)^{p+\beta}} \frac{J'(u_x,\rho(x))\rho(x)}{J(u_x,\rho(x))} dV_g.
\end{align}
Plugging \eqref{eq:onetwoadd} into \eqref{eq:onetwopolar}, we obtain
\begin{align*}
\int_M &\frac{|f|^p}{\rho(x)^{p+\beta}} dV_g\notag\\
&=\frac{p}{n(p-1)+\beta} \int_M \frac{|f|^{p-2} f}{\rho(x)^{\frac{p-1}p(p+ \beta)}} \frac{\pa_\rho f + (\frac{n-1}{\rho(x)} + \frac{J'(u_x,\rho(x))}{J(u_x,\rho(x))})f}{\rho(x)^{\frac\beta p}}dV_g\notag\\
&\quad -\frac{p-1}{n(p-1)+ \beta} \int_M \frac{|f|^p}{\rho(x)^{p+\beta}} \frac{J'(u_x,\rho(x))\rho(x)}{J(u_x,\rho(x))} dV_g\notag\\
&= \frac1p \lt(\frac{p}{n(p-1)+\beta}\rt)^p \int_M \frac{|\pa_\rho f + (\frac{n-1}{\rho(x)} + \frac{J'(u_x,\rho(x))}{J(u_x,\rho(x))})f|^p}{\rho(x)^{\beta}} dV_g+ \frac{p-1}p \int_{M}\frac{|f|^p}{\rho(x)^{p+\beta}} dV_g\notag\\
&\quad  -\int_M \frac{R_p\lt(\frac f{\rho(x)},\frac{p}{n(p-1)+\beta} (\pa_\rho f + (\frac{n-1}{\rho(x)} + \frac{J'(u_x,\rho(x))}{J(u_x,\rho(x))})f)\rt)}{\rho(x)^{\beta}} dV_g\notag\\
&\quad -\frac{p-1}{n(p-1)+ \beta} \int_M \frac{|f|^p}{\rho(x)^{p+\beta}} \frac{J'(u_x,\rho(x))\rho(x)}{J(u_x,\rho(x))} dV_g,
\end{align*}
which is equivalent to
\begin{align}\label{eq:onetwoidentity}
\int_M \frac{|f|^p}{\rho(x)^{p+\beta}} dV_g&= \lt(\frac{p}{n(p-1)+\beta}\rt)^p \int_M \frac{|\pa_\rho f + (\frac{n-1}{\rho(x)} + \frac{J'(u_x,\rho(x))}{J(u_x,\rho(x))})f|^p}{\rho(x)^{\beta}} dV_g\notag\\
&\quad -p\int_M \frac{R_p\lt(\frac f{\rho(x)},\frac{p}{n(p-1)+\beta} (\pa_\rho f + (\frac{n-1}{\rho(x)} + \frac{J'(u_x,\rho(x))}{J(u_x,\rho(x))})f)\rt)}{\rho(x)^{\beta}} dV_g\notag\\
&\quad -\frac{p(p-1)}{n(p-1)+ \beta} \int_M \frac{|f|^p}{\rho(x)^{p+\beta}} \frac{J'(u_x,\rho(x))\rho(x)}{J(u_x,\rho(x))} dV_g.
\end{align}
Since $\beta > -n(p-1)$, $R_p\geq 0$ and $J'(u,\rho) \geq 0$, then the inequality \eqref{eq:onetwoinequality} is an immediate consequence of \eqref{eq:onetwoidentity}. It remains to check the sharpness of \eqref{eq:onetwoinequality}. For $0 < \de < 1/2$, define
\[
f_\de(x) = \vphi(\rho(x)) (1 -\vphi(\de^{-1}\rho(x))) \rho(x)^{-\frac{n-p-\beta}p},
\]
where $\vphi$ is cut-off function in $(-1,1)$. An easy computation shows that
\[
\int_M \frac{f_\de(x)^p}{\rho(x)^{p+\de}} dV_g = \int_{S^{n-1}}\int_{\frac\de2}^{1} \vphi(\rho)^p (1-\vphi(\de^{-1}\rho))^p J(u,\rho) \rho^{-1} d\rho du \geq |S^{n-1}| \ln (2\de)^{-1}.
\]
Hence
\[
\lim_{\de \to 0}\frac{f_\de(x)^p}{\rho(x)^{p+\de}} dV_g = \infty.
\]
Obviously, we have
\begin{align*}
\pa_\rho f + &\lt(\frac{n-1}{\rho(x)} + \frac{J'(u_x,\rho(x))}{J(u_x,\rho(x))}\rt)f\\
&= \vphi'(\rho) \rho^{-\frac{n-p-\beta}p} -\frac1\de \vphi'(\de^{-1}\rho)\rho^{-\frac{n-p-\beta}p} + \frac{J'}{J}\vphi(\rho)(1 -\vphi(\de^{-1}\rho)) \rho^{-\frac{n-p-\beta}p}\\
&\quad + \frac{n(p-1)+\beta}p \vphi(\rho)(1 -\vphi(\de^{-1}\rho)) \rho^{-\frac{n-\beta}p}.
\end{align*}
We can readily check that
\[
\int_M \frac{|\vphi'(\rho) \rho^{-\frac{n-p-\beta}p}|^p}{\rho(x)^{\beta}} dV_g = O(1),
\]
\[
\int_M \frac{|\frac1\de \vphi'(\de^{-1}\rho)\rho^{-\frac{n-p-\beta}p}|^p}{\rho(x)^{\beta}} dV_g = O(1),
\]
\[
\int_M \frac{|\frac{J'}{J}\vphi(\rho)(1 -\vphi(\de^{-1}\rho)) \rho^{-\frac{n-p-\beta}p}|^p}{\rho(x)^{\beta}} dV_g = O(1),
\]
and
\[
\int_M \frac{|\vphi(\rho)(1 -\vphi(\de^{-1}\rho)) \rho^{-\frac{n-\beta}p}|^p}{\rho(x)^{\beta}} dV_g = \int_M\frac{f_\de(x)^p}{\rho(x)^{p+\de}} dV_g.
\]
Therefore,
\[
\lim_{\de\to 0^+} \frac{\int_M \frac{|\pa_\rho f + \lt(\frac{n-1}{\rho(x)} + \frac{J'(u_x,\rho(x))}{J(u_x,\rho(x))}\rt)f|^p}{\rho(x)^{p+\beta}}dV_g}{\int_M\frac{f_\de(x)^p}{\rho(x)^{p+\de}} dV_g} = \lt(\frac{n(p-1)+ \beta}p\rt)^p.
\]
This finishes our proof.
\end{proof}

If $K_M \leq -b \leq 0$, then the identity \eqref{eq:onetwoidentity} implies a quantitative version of \eqref{eq:onetwoinequality} as follows
\begin{align}\label{eq:quanti12}
\int_M &\frac{|\pa_\rho f + (\frac{n-1}{\rho(x)} + \frac{J'(u_x,\rho(x))}{J(u_x,\rho(x)}) f|^p}{\rho(x)^\beta} dV_g \notag\\
&\qquad\qquad\geq \lt(\frac{n(p-1)+\beta}p\rt)^p\int_M \frac{|f|^p}{\rho(x)^{p+\beta}} dV_g\notag\\
&\qquad\qquad \quad +3b(n-1)(p-1) \lt(\frac{n(p-1)+\beta} p\rt)^{p-1} \int_{M} \frac{|f|^p}{\rho^{p+ \beta -2}(\pi^2 + b \rho^2)} dV_g.
\end{align}

Replacing $f$ by $\pa_\rho f$ in Lemma \ref{onetwolemma} and \eqref{eq:quanti12}, we obtain the following Rellich type inequality which connects first to second order derivatives.

\begin{theorem}\label{Rellichonetwo}
Let $(M,g)$ be an $n-$dimensional Cartan--Hadamard manifold. Suppose that $n \geq 2$, $p\in (1,n)$ and $-n(p-1)< \beta < n-p$. There holds for any $f \in C_0^\infty(M)$
\begin{equation}\label{eq:Rellichonetwo}
\int_M \frac{|\Delta_{g,\rho} f|^p}{\rho(x)^\beta} dV_g \geq \lt(\frac{n(p-1)+\beta}p\rt)^p \int_M \frac{|\pa_\rho f|^p}{\rho(x)^{p+\beta}} dV_g.
\end{equation}
Furthermore, the constant $(\frac{n(p-1)+\beta}p)^p$ is the best constant in \eqref{eq:Rellichonetwo}.

If $K_M \leq -b \leq 0$, then we have
\begin{align}\label{eq:quantiRellichonetwo}
\int_M \frac{|\Delta_{g,\rho} f|^p}{\rho(x)^\beta} dV_g &\geq \lt(\frac{n(p-1)+\beta}p\rt)^p \int_M \frac{|\pa_\rho f|^p}{\rho(x)^{p+\beta}} dV_g\notag\\
&\quad + 3b(n-1)(p-1) \lt(\frac{n(p-1)+\beta} p\rt)^{p-1} \int_{M} \frac{|\pa_\rho f|^p}{\rho^{p+ \beta -2}(\pi^2 + b \rho^2)} dV_g.
\end{align}
\end{theorem}

Combining \eqref{eq:Rellichonetwo}, \eqref{eq:quantiRellichonetwo}, \eqref{eq:subcriticalH} and \eqref{eq:quantitativeHb1}, we get the following weighted Rellich inequalities on $M$.

\begin{theorem}\label{Rellich}
Let $M$ be an $n-$dimensional Cartan--Hadamard manifold. Suppose that $n \geq 3$, $p\in (1,n/2)$ and $-n(p-1)< \beta < n-2p$. There holds for any $f \in C_0^\infty(M)$
\begin{equation}\label{eq:Rellich}
\int_M \frac{|\Delta_{g,\rho} f|^p}{\rho(x)^\beta} dV_g \geq \lt(\frac{(n(p-1)+\beta)(n-2p-\beta)}{p^2}\rt)^p\int_M \frac{|f|^p}{\rho(x)^{2p+\beta}} dV_g.
\end{equation}
Furthermore, the constant $(\frac{(n(p-1)+\beta)(n-2p-\beta)}{p^2})^p$ is the best constant in \eqref{eq:Rellich}.

If $K_M \leq -b \leq 0$, then we have
\begin{align}\label{eq:quantiRellich}
\int_M &\frac{|\Delta_{g,\rho} f|^p}{\rho(x)^\beta} dV_g\notag\\
 &\quad\geq \lt(\frac{(n(p-1)+\beta)(n-2p-\beta)}{p^2}\rt)^p\int_M \frac{|f|^p}{\rho(x)^{2p+\beta}} dV_g\notag\\
&\quad\quad + 3b(n-1)(p-1) \lt(\frac{n(p-1)+\beta} p\rt)^{p-1} \int_{M} \frac{|\pa_\rho f|^p}{\rho^{p+ \beta -2}(\pi^2 + b \rho^2)} dV_g\notag\\
&\quad\quad + 3b(n-1)\lt(\frac{n-2p-\beta}p\rt)^{p-1} \lt(\frac{n(p-1)+ \beta}p\rt)^p \int_M \frac{|f|^p}{\rho^{2p+\beta -2}(\pi^2 + b\rho^2)} dV_g.
\end{align}
\end{theorem}
\begin{proof}
Since $-n(p-1) < p+ \beta < n-p$, then by the weighted Hardy inequality \eqref{eq:subcriticalH}, we have
\[
\int_{M} \frac{|f|^p}{\rho^{2p+\beta}} dV_g \leq \lt(\frac p{n-2p-\beta}\rt)^p \int_M \frac{|\pa_\rho f|^2}{\rho^{p+\beta}} dV_g.
\]
In the other hand, by \eqref{eq:Rellichonetwo} we get
\[
\int_M \frac{|\pa_\rho f|^2}{\rho^{p+\beta}} dV_g \leq  \lt(\frac{p}{n(p-1)+\beta}\rt)^p \int_M \frac{|\Delta_{g,\rho }f|^p}{\rho(x)^\beta} dV_g.
\]
Combining these two estimates, we obtain \eqref{eq:Rellich}. To check the sharpness of \eqref{eq:Rellich}, we use the approximation of $\rho^{-(n-2p-\beta)/p}$ as follows
\[
f_\de(x) = \vphi(\rho(x)) (1-\vphi(\de^{-1}\rho(x))) \rho(x)^{-\frac{n-2p-\beta}p}
\]
where $\vphi$ is cut-off function in $C^\infty_0((-1,1))$ and $0 < \de < 1/2$. Using the same argument as in the proof of Lemma \ref{onetwolemma} by making the straightforward (but tedious) compuations, we can show that
\[
\lim_{\de\to 0^+} \frac{\int_M \frac{|\Delta_{g,\rho} f_\de|^p}{\rho(x)^\beta} dV_g}{\int_{M} \frac{|f|^p}{\rho^{2p+\beta}} dV_g} =\lt(\frac{(n-2p-\beta)(n(p-1)+\beta)}{p^2}\rt)^p
\]
which implies the sharpness of \eqref{eq:Rellich}.

The proof of \eqref{eq:quantiRellich} is completely similar by iterating \eqref{eq:quantitativeHb1} and \eqref{eq:quantiRellichonetwo}.
\end{proof}

We next consider the critical case $\beta = n-2p$. In this case, we obtain a critical Rellich inequality which generalizes the inequality \eqref{eq:criticalH} to order two.

\begin{theorem}\label{criticalRellich}
Let $(M,g)$ be an $n-$dimensional Cartan--Hadamard manifold. Suppose that $n \geq 3$, $p\in (1,n)$. There holds for any $f \in C_0^\infty(B_1(P))$
\begin{equation}\label{eq:criticalRellich}
\int_{B_1(P)} \frac{|f(x)|^p}{\rho(x)^{n}(\ln \frac1{\rho(x)})^p} dV_g \leq \lt(\frac{p}{(p-1)(n-2)}\rt)^p \int_{B_1(P)} \frac{|\Delta_{g,\rho} f|^p}{\rho(x)^{n-2p}} dV_g.
\end{equation}
Furthermore, the constant $(\frac p{(n-2)(p-1)})^p$ is the best constant in \eqref{eq:criticalRellich}.

If $K_M \leq -b \leq 0$, then we have
\begin{align}\label{eq:quantiCR}
\int_{B_1(P)} &\frac{|\Delta_{g,\rho} f|^p}{\rho^{n-2p}} dV_g \notag\\
&\quad\quad\geq \lt(\frac{(p-1)(n-2)}p\rt)^p\int_{B_1(P)} \frac{|f(x)|^p}{\rho(x)^{n}(\ln \frac1{\rho(x)})^p} dV_g\notag\\
&\quad\quad\quad + 3b(n-1)(p-1) (n-2)^{p-1}\int_{B_1(P)} \frac{|\pa_\rho f|^p}{\rho^{n-p-2}(\pi^2+ b\rho^2)}dV_g\notag\\
&\quad\quad\quad + 3b(n-1)(n-2)^p\lt(\frac{p-1}{p}\rt)^{p-1} \int_{B_1(P)}\frac{|f|^p}{\rho^{n-2} (\ln \frac1\rho)^{p-1}(\pi^2 + b\rho^2)} dV_g.
\end{align}
\end{theorem}
\begin{proof}
The inequality \eqref{eq:criticalRellich} is consequence of \eqref{eq:criticalH} and \eqref{eq:Rellichonetwo} with $\beta = n-2p$. Note that the condition $-n(p-1) < \beta < n-p$ holds true since $n\geq 3$. To check the sharpness of \eqref{eq:criticalRellich}, we use the following sequence of test functions
\[
f_\de(x) =\lt(\ln \frac 1{\rho(x)}\rt)^{\frac{p-1}p -\de} \vphi(\rho(x)),
\]
where $\vphi$ is cut-off function in $(-1,1)$. Making the computations as in the proof of Theorem \ref{CH}, we obtain the desire result.

The inequality \eqref{eq:quantiCR} is followed from \eqref{eq:QuantiCHb1} and \eqref{eq:quantiRellichonetwo}.
\end{proof}

Iterating the weighted Hardy and Rellich inequalities (both in the subcritical and critical cases), we obtain the following weighted Rellich inequality for higher order derivatives (both in the subcritical and critical cases respectively) on $M$. The detail proof is left to the readers. Let us denote
\[
c(n,2l,\beta,p) = \lt(\prod_{i=0}^{l-1} \frac{p^2}{(n-2p-\beta-2ip)(n(p-1)+\beta +2ip)}\rt)^p
\]
for $l\geq 1$, $p\in (1, n/(2l))$ and $-n(p-1) < \beta < n-2lp$.
\begin{theorem}\label{Rellichhigher}
Let $M$ be an $n-$dimensional Cartan--Hadamard manifold and let $k$ be a positive integer. Suppose that $n \geq 3$, and $p\in (1,n/k)$. Then for any function $f \in C_0^\infty(M)$ the following inequalities hold true.
\begin{description}
\item (i) If $k =2l$, $l\geq 1$ and $n(1-p) < \beta < n-2lp$, then we have
\begin{equation}\label{eq:even}
\int_{M} \frac{|f|^p}{\rho(x)^{2lp + \beta}} dV_g \leq c(n,2l,\beta,p) \int_M \frac{|\De_{g,\rho}^lf|^p }{\rho(x)^\beta} dV_g,
\end{equation}
and if $K_M \leq -b \leq 0$ then 
\begin{align}\label{eq:quantieven}
\int_M \frac{|\De_{g,\rho}^lf|^p }{\rho(x)^\beta} dV_g &\geq \frac1{c(n,2l,\beta,p)}\int_{M} \frac{|f|^p}{\rho(x)^{2lp + \beta}} dV_g \notag\\
&\quad + \frac{3b(n-1)p}{(n-2lp-\beta)c(n,2l,\beta,p)} \int_M \frac{|f|^p}{\rho^{2lp+\be -2}(\pi^2 + b \rho^2)} dV_g.
\end{align}
\item (ii) If $k =2l+1$, $l\geq 1$ and $n-n(p+1) < \beta < n-(2l+1)p$ then we have
\begin{equation}\label{eq:odd}
\int_{M} \frac{|f|^p}{\rho(x)^{(2l+1)p + \beta}} dx \leq \frac{p^p}{(n-p-\beta)^p}c(n,2l,p+\beta,p)\int_M \frac{|\pa_\rho \De_{g,\rho}^l f|^p}{\rho(x)^\beta} dV_g.
\end{equation}
and if $K_M \leq -b \leq 0$ then 
\begin{align}\label{eq:quantiodd}
&\frac{p^p}{(n-p-\beta)^p}c(n,2l,p+\beta,p)\int_M \frac{|\pa_\rho\De_{g,\rho}^lf|^p }{\rho(x)^\beta} dV_g \notag\\
&\quad \geq \int_{M} \frac{|f|^p}{\rho^{(2l+1)p + \beta}} dV_g  + \frac{3b(n-1)p}{(n-(2l+1)p-\beta)} \int_M \frac{|f|^p}{\rho^{(2l+1)p+\be -2}(\pi^2 + b \rho^2)} dV_g.
\end{align}
\end{description}
Furthermore, the inequalities \eqref{eq:even} and \eqref{eq:odd} are sharp.
\end{theorem}
For the critical case $\beta =n -kp$, we have the following critical Rellich inequalities on $M$ which generalize Theorem \ref{CH} and Theorem \ref{criticalRellich} to higher order derivatives.

\begin{theorem}\label{criticalRellichhigher}
Let $(M,g)$ be an $n-$dimensional Cartan--Hadamard manifold and let $k$ be a positive integer. Suppose that $n \geq 3$ and $p\in (1,n/k)$. Then for any function $f \in C_0^\infty(B_1(P))$ the following inequalities hold true.
\begin{description}
\item (i) If $k =2l$, $l\geq 1$ then we have
\begin{equation}\label{eq:criticaleven}
\int_{B_1(P)} \frac{|f|^p}{\rho^{n} \lt(\ln \frac1{\rho}\rt)^p} dV_g \leq \lt( p'\frac{2^{1-l}}{(l-1)!}\prod_{i=0}^{l-1} \frac1{n -2i-2}\rt)^p \int_{B_1(P)} \frac{|\De_{g,\rho}^lf|^p }{\rho^{n-2lp}} dV_g,
\end{equation}
here $p' = p/(p-1)$, and if $K_M \leq -b \leq 0$ then we have
\begin{align}\label{eq:QCReven}
&\lt( p'\frac{2^{1-l}}{(l-1)!}\prod_{i=0}^{l-1} \frac1{n -2i-2}\rt)^p \int_{B_1(P)} \frac{|\De_{g,\rho}^lf|^p }{\rho^{n-2lp}} dV_g\notag\\
&\quad \geq \int_{B_1(P)} \frac{|f|^p}{\rho^{n} \lt(\ln \frac1{\rho}\rt)^p} dV_g + \frac{3b(n-1)p}{p-1} \int_{B_1(P)}\frac{|f|^p}{\rho^{n-2} (\ln \frac1\rho)^{p-1}(\pi^2 + b\rho^2)} dV_g.
\end{align}
\item (ii) If $k =2l+1$, $l\geq 1$ then we have
\begin{equation}\label{eq:criticalodd}
\int_{B_1(P)} \frac{|f|^p}{\rho^{n}\lt(\ln \frac1{\rho}\rt)^p} dx \leq \lt(p'\frac1{2^l l!}\prod_{i=0}^{l-1} \frac1{n -2i-2}\rt)^p \int_{B_1(P)} \frac{|\pa_\rho \De_{g,\rho}^l f|^p}{\rho^{n-(2l+1)p}} dV_g,
\end{equation}
and if $K_M \leq -b \leq 0$ then we have
\begin{align}\label{eq:QCRodd}
&\lt(p'\frac1{2^l l!}\prod_{i=0}^{l-1} \frac1{n -2i-2}\rt)^p \int_{B_1(P)} \frac{|\pa_\rho \De_{g,\rho}^l f|^p}{\rho^{n-(2l+1)p}} dV_g\notag\\
&\quad\geq \int_{B_1(P)} \frac{|f|^p}{\rho^{n}\lt(\ln \frac1{\rho}\rt)^p} dx + \frac{3b(n-1)p}{p-1} \int_{B_1(P)}\frac{|f|^p}{\rho^{n-2} (\ln \frac1\rho)^{p-1}(\pi^2 + b\rho^2)} dV_g.
\end{align}
\end{description}
Furthermore, the inequalities \eqref{eq:criticaleven} and \eqref{eq:criticalodd} are sharp.
\end{theorem}
We emphasize here that in the Euclidean space $M =\R^n$, Theorems \ref{Rellichonetwo}, \ref{Rellich}, \ref{criticalRellich}, \ref{Rellichhigher} and \ref{criticalRellichhigher} was recently proved by the author \cite{Nguyen2017} (The same inequalities on radial functions was previously proved by Adimurthi et al. \cite{AGS2006} and by Adimurthi and Santra \cite{AS2009}). More precisely, in \cite{Nguyen2017}, the author also proved the generalizations of these inequalities on $\R^n$ to more general class of homogeneous groups equipped with any homogeneous quasi-norm with the same best constants.

\section{Hardy and Rellich inequalities in hyperbolic spaces}
We conclude this section by giving some concrete examples on the $n-$dimensional hyperbolic spaces $\mathbb H^n$. We will use the Poincar\'e conformal disc model for hyperbolic spaces $\mathbb H^n$, i.e.,  the underlying space which we consider is the unit ball 
\[
\mathbb B_n =\{x =(x_1,\ldots,x_n)\in \R^n\, :\, |x| = \sqrt{x_1^2 + \cdots + x_n^2} < 1\}
\]
equipped with metric 
\[
g(x) = \lt(\frac2{1-|x|^2}\rt)^n dx.
\]
The volume element on $\mathbb B_n$ is given by $dV = \lt(\frac2{1-|x|^2}\rt)^n dx$ and the associated Laplace--Beltrami operator is given by
\[
\De_g = \frac{(1-|x|^2)^2}4 \lt(\sum_{i=1}^n \frac{\pa^2}{\pa x_i^2} + 2(n-2)\sum_{i=1}^n \frac{x_i}{1-|x|^2} \frac{\pa}{\pa x_i}\rt)
\]
and the corresponding gradient is 
\[
\na_g = \lt(\frac{1-|x|^2}2\rt)^2 \lt(\frac{\pa}{\pa x_1},\ldots, \frac{\pa}{\pa x_n}\rt).
\]
The geodesic distance from $x$ to $0$ is $\rho(x) = \ln \frac{1+|x|}{1-|x|}$. Finally, recall that $K_{\mathbb B_n} \equiv -1$.

Our main results in this section give us several quantitative Hardy type inequalities on the hyperbolic spaces $\mathbb B_n$ as follows.

\begin{theorem}\label{Hyperbolic}
Suppose $n\geq 2$, $p\in (1,n)$ and $\beta < n-p$. Then there exists $c >0$ such that the following inequality holds for any $f \in C_0^\infty(\mathbb B_n)$
\begin{align}\label{eq:hyper}
\int_{\mathbb B_n} \frac{|\na_g f|^p}{\rho^\beta} dV_g&\geq \lt(\frac{n-p-\beta}p\rt)^p \int_{\mathbb B_n} \frac{|f|^p}{\rho^{p+ \beta}}dV_g\notag\\
&\quad + 3c(n-1) \lt(\frac{n-p-\beta}p\rt)^{p-1}\int_{\mathbb B_n} \frac{|f|^p}{\rho^{p+ \beta-2}} dx.
\end{align}
Let $B_1(0)$ denote the geodesic unit ball with center at $0$ in $\mathbb B_n$. There exists a constant $c>0$ such that the following inequality holds for any $p >1$ and $f\in C_0^\infty(B_1(0))$
\begin{align}\label{eq:hypera}
\int_{B_1(0)} \frac{|\na_g f|_g^p}{\rho^{n-p}} dV_g &\geq \lt(\frac{p-1}p\rt)^p \int_{B_1(0)}\frac{|f|^p}{\rho^n (\ln \frac1\rho)^p} dV_g\notag\\
&\quad +3C(n-1) \lt(\frac{p-1}p\rt)^{p-1} \int_{B_1(0)}\frac{|f|^p}{\rho^{n-2} (\ln \frac1\rho)^{p-1}}dx.
\end{align}
\end{theorem}
\begin{proof}
From the formula for $\rho(x)$, we have $|x| = \frac{e^{\rho} -1}{e^{\rho}+1}$. Hence
\[
\lt(\frac1{1-|x|^2}\rt)^n = \lt(\frac{(1+e^\rho)^2}{4e^\rho}\rt)^n \geq C(\pi^2 + \rho^2)
\]
for some $C >0$. Thus Theorem \ref{Hyperbolic} follows from \eqref{eq:quantitativeHb}, \eqref{eq:QuantiCHb} and the previous inequality.
\end{proof}
Note that \eqref{eq:hyper} extends a result of Kombe and \"Ozaydin (see \cite[Theorem $3.1$]{KO2009}) in the case $p =2$ to any $p\in (1,n)$.

We next consider the Rellich inequalities. Denote by $\pa_r = \frac{x}{|x|} \cdot \na$ the radial derivative in $\R^n$, and denote by
\[
\De_r = \pa_r^2 + \frac{n-1} r\pa_r
\]
the radial Laplace in $\R^n$. It was proved by Machihara, Ozawa and Wadade \cite{MOW2017} that
\[
\int_{\R^n} |\De_r u|^2 dx \leq \int_{\R^n} |\De u|^2 dx
\]
for any $u\in C_0^\infty(\R^n)$. Our next results shows that such a result also holds true on hyperbolic space $\mathbb H^n$ (even in the weighted form).

\begin{theorem}\label{MOWtype}
Let $n \geq 3$ and $-2 < \be \leq n-4$. It holds
\begin{equation}\label{eq:MOWtypeonH}
\int_{\mathbb B_n} |\De_{g,\rho} f|^2 \rho^{-\beta} dV_g \leq \int_{\B_n} |\De_g f|^2 \rho^{-\beta} dV_g,\quad f\in C_0^\infty(\B_n).
\end{equation}
Furthermore, equality holds in \eqref{eq:MOWtypeonH} if and only if $f$ is radial function.
\end{theorem}
\begin{proof}
Denote $r =|x| =\frac{e^{\rho} -1}{e^{\rho} +1}$, then 
\[
\De_g = \frac{(1-r^2)^2}4 \De + (n-2)\frac{1-r^2}{2} r \pa_r,\quad \De_{g,\rho} = \frac{(1-r^2)^2}4 \De_r + (n-2)\frac{1-r^2}{2} r\pa_r.
\]
For any function $f \in C_0^\infty(\B_n)$, we decompose it into spherical harmonic as
\begin{equation}\label{eq:spherical}
f(x) = \sum_{k=0}^\infty f_k(r) \phi_k(\si),\quad \si \in S^{n-1},\, x = r\si,
\end{equation}
where $\phi_k$ is orthonormal eigenfunction of the Laplace--Beltrami operator on the sphere $S^{n-1}$ with respect to eigenvalue $c_k = k(n+k-2)$ with $k =0,1,2,\ldots$. The function $f_k$ belongs to $C_0^\infty(\B_n)$ and satifies $f_k(r) = O(r^k)$, $f'_k(r) =O(r^{k-1})$ as $r \downarrow 0$. In particular, we have $\phi_0\equiv 1$ and $f_0(r) = \frac{1}{n\om_n}\int_{S^{n-1}} f(r\si) d\si$. From the decomposition of $f$, we have
\[
\De_{g,\rho} f(x) = \sum_{k=0}^\infty \De_{g} f_k(r)\, \phi_k(\si),
\]
and
\[
\De_g f(x) = \sum_{k=0}^\infty\lt(\De_{g} f_k(r) -c_k \frac{(1-r^2)^2}4 \frac{f_k(r)}{r^2}\rt)\phi_k(\si).
\]
Thus, to prove \eqref{eq:MOWtypeonH}, it's enough to verify that
\begin{equation}\label{eq:abcd}
c_k\int_{\B_n} \frac{f_k^2}{\rho^\be} \lt(\frac{1-r^2}{2r}\rt)^4 dV_g -2\int_{\B_n} (f_k \De_g f_k) \, \frac1{\rho^\beta}\lt(\frac{1-r^2}{2r}\rt)^2 dV_g \geq 0,\qquad k\geq 1.
\end{equation}
Note that $2 f_k \De_g f_k = \De_g f_k^2 -2 |\na_g f_k|_g^2$. Hence, by using integration by parts, \eqref{eq:abcd} is equivalent to
\begin{multline}\label{eq:abcd1}
c_k\int_{\B_n} \frac{f_k^2}{\rho^\be} \lt(\frac{1-r^2}{2r}\rt)^4 dV_g +2\int_{\B_n} |\na_g f_k|_g^2  \frac1{\rho^\beta}\lt(\frac{1-r^2}{2r}\rt)^2 dV_g\\
 -\int_{\B_n} f_k^2 \De_g\lt(\frac1{\rho^\beta}\lt(\frac{1-r^2}{2r}\rt)^2\rt) dV_g \geq 0,
\end{multline}
for any $k \geq 1$. Remark that $r =\frac{e^\rho-1}{e^\rho +1}$ and hence 
\[
\frac1{\rho^\beta}\lt(\frac{1-r^2}{2r}\rt)^2 = \frac1{\rho^\be \sinh^2 \rho} =: k(\rho).
\] 
From \eqref{eq:Lapofradial}, we have
\begin{align}\label{eq:Lap}
\De_g k(\rho) &= k''(\rho) + (n-1) \frac{\cosh \rho}{\sinh \rho} k'(\rho)\notag\\
&=- k(\rho)\lt(2 +2(n-4)\frac{\cosh^2\rho}{\sinh^2\rho}-  \be(\be +1) \frac1{\rho^{2}} +\beta(n-5) \frac{\cosh \rho}{\rho \sinh \rho} \rt).
\end{align}

We next show that 
\begin{equation}\label{eq:sosanh}
\int_{\B_n} |\na_g u|_g^2  \frac1{\rho^\beta}\lt(\frac{1-r^2}{2r}\rt)^2 dV_g \geq \frac{(n-\beta -4)^2}4\int_{\B_n} \frac{u^2}{\rho^\be} \lt(\frac{1-r^2}{2r}\rt)^4 dV_g,
\end{equation}
for any radial function $u\in C_0^\infty(\B_n)$. Define the function $F$ on $[0,\infty)$ by
\[
F(\rho) = u(r), \quad r = \frac{e^\rho-1}{e^\rho +1}.
\]
Then, \eqref{eq:sosanh} is equivalent to
\begin{equation}\label{eq:sosanh1}
\int_0^\infty (F'(\rho))^2\rho^{n-\be-3} J_1(\rho)^{\frac{n-3}{n-1}} d\rho \geq \frac{(n-\beta -4)^2}4\int_0^\infty F(\rho)^2 \rho^{n-\beta -5} J_1(\rho)^{\frac{n-5}{n-1}} d\rho,
\end{equation}
Recall that $J_1(\rho) = (\frac{\sinh \rho}{\rho})^{n-1}$. Indeed, using integration by parts and $\be < n-4$ we get
\begin{align*}
\int_0^\infty F(\rho)^2 \rho^{n-\beta -5} J_1(\rho)^{\frac{n-5}{n-1}} d\rho& = -\frac{2}{n-\beta -4} \int_0^\infty F(\rho) F'(\rho) \rho^{n-\beta -4} J_1(\rho)^{\frac{n-5}{n-1}} d\rho\\
&\quad - \frac{1}{n-\beta -4} \int_0^\infty F(\rho)^2 \rho^{n-\beta -4} J_1'(\rho)J_1(\rho)^{-\frac{4}{n-1}} d\rho.
\end{align*}
Applying H\"older inequality and using $J_1' \geq 0$, $J_1 \geq 1$, we get
\begin{align*}
\int_0^\infty F(\rho)^2 \rho^{n-\beta -5} J_1(\rho)^{\frac{n-5}{n-1}} d\rho& \leq \frac{4}{(n-4-\beta)^2} \int_0^\infty (F'(\rho))^2 \rho^{n-\beta -3}J_1(\rho)^{\frac{n-5}{n-1}} d\rho\\
&\leq \frac{4}{(n-4-\beta)^2} \int_0^\infty (F'(\rho))^2 \rho^{n-\beta -3}J_1(\rho)^{\frac{n-3}{n-1}} d\rho
\end{align*}
which implies \eqref{eq:sosanh1}. Consequently, the left hand side of \eqref{eq:abcd1} is at least
\begin{align*}
&\lt(c_k + \frac{(n-\be -4)^2}2\rt)\int_{\B_n} \frac{f_k^2}{\rho^\be} \lt(\frac{1-r^2}{2r}\rt)^4 dV_g \\
&\hspace{2cm} + \int_{\B_n}\frac{f_k^2}{\rho^\be} \lt(\frac{1-r^2}{2r}\rt)^4 \Big(2\sinh^2\rho +2(n-4)\cosh^2\rho - &\\
&\hspace{6.2cm}-  \be(\be +1) \frac{\sinh^2\rho}{\rho^{2}} +\beta(n-5) \frac{\sinh(2\rho)}{2\rho} \Big) dV_g.
\end{align*}
Hence, to prove \eqref{eq:abcd1}, it is enough to show
\begin{equation}\label{eq:abcd2}
c_k + \frac{(n-\be -4)^2}2 + 2\sinh^2\rho +2(n-4)\cosh^2\rho -  \be(\be +1) \frac{\sinh^2\rho}{\rho^{2}} +\beta(n-5) \frac{\sinh(2\rho)}{2\rho} \geq 0, 
\end{equation}
for $\rho >0$. It is suffice to check \eqref{eq:abcd2} for $k=1$. Expanding the exponent function in series form, the left hand side of \eqref{eq:abcd2} is equal to
\[
\frac{(n-2+k)^2+k^2-(\beta +2)^2}2 + \sum_{l=1}^{\infty}\lt((n-3) -\frac{\be(\be+1)}{(l+1)(2l+1)} + \frac{\be(n-5)}{2l+1} \rt)\frac{(2\rho)^{2l}}{(2l)!}.
\]
Since $k\geq 1$ and $-2 < \be \leq n-4$, we can easily check that
\begin{equation}\label{eq:sosanhheso}
\frac{(n-2+k)^2+k^2-(\beta +2)^2}2 >0,\quad\text{and}\quad (n-3) -\frac{\be(\be+1)}{(l+1)(2l+1)} + \frac{\be(n-5)}{2l+1} \geq 0,
\end{equation}
for any $l\geq 1$. This finishes the proof of \eqref{eq:MOWtypeonH}.

Suppose that equality holds true in \eqref{eq:MOWtypeonH} for some function $f$. Expanding $f$ in spherical harmonic expression as in \eqref{eq:spherical}. By \eqref{eq:sosanhheso}, we must have $c_k =0$ for any $k\geq 1$. This shows that $f$ is radial function.
\end{proof}

By Theorem \ref{MOWtype}, we see that in the hyperbolic space $\mathbb H^n$, the Rellich inequality \eqref{eq:Rellich} with $p=2$ is stronger than the inequality of Kombe and \"Ozaydin \cite{KO2009}: suppose $-2 < \beta < n-4$
\begin{equation}\label{eq:KO}
\frac{(n+ \beta)^2(n-4 -\be)^2}{16}\int_{\B_n} \frac{|f|^2}{\rho^{\be + 4}} dV_g \leq \int_{\B_n} \frac{|\De_g f|^2}{\rho^\beta} dV_g, \quad f \in C_0^\infty(\B_n).
\end{equation}
Similarly, \eqref{eq:quantiRellich} implies an improvements of \eqref{eq:KO}
\begin{align}\label{eq:improve}
\int_{\B_n} \frac{|\De_g f|^2}{\rho^\beta} dV_g& \geq \frac{(n+ \beta)^2(n-4 -\be)^2}{16}\int_{\B_n} \frac{|f|^2}{\rho^{\be + 4}} dV_g + 3(n-1)\frac{n+\be}2 \int_{\B_n} \frac{|\pa_\rho f|^2}{\rho^\be(\pi^2 + \rho^2)} dV_g\notag\\
&\quad + 3(n-1) \frac{(n-4-\beta)(n+\be)^2}8\int_{\B_n} \frac{|f|^2}{\rho^{\be+2}(\pi^2 + \rho^2)} dV_g.
\end{align}
It is easy to prove that
\begin{equation}\label{eq:Aq}
\int_{\B_n} \frac{|\pa_\rho f|^2}{\rho^\be(\pi^2 + \rho^2)} dV_g \geq \frac{(n-\be -4)^2}{4} \int_{\B_n} \frac{|f|^2}{\rho^{\be+2}(\pi^2 + \rho^2)} dV_g.
\end{equation}
Combining \eqref{eq:improve} and \eqref{eq:Aq} yields
\begin{align}\label{eq:improveRsao}
\int_{\B_n} \frac{|\De_g f|^2}{\rho^\beta} dV_g& \geq \frac{(n+ \beta)^2(n-4 -\be)^2}{16}\int_{\B_n} \frac{|f|^2}{\rho^{\be + 4}} dV_g\notag\\
&\quad + 3 \frac{(n-1)(n-2)(n+\be)(n-4-\be)}4\int_{\B_n} \frac{|f|^2}{\rho^{\be+2}(\pi^2 + \rho^2)} dV_g.
\end{align}
Using the simple inequality in the proof of Theorem \ref{Hyperbolic}, we prove the following improved Rellich inequality in $\mathbb H^n$.
\begin{theorem}\label{ImprovedR}
Suppose that $n\geq 4$ and $-2 < \beta < n-4$. Then there exists  a constant $C >0$ such that
\begin{align}\label{eq:improveR}
\int_{\B_n} \frac{|\De_g f|^2}{\rho^\beta} dV_g& \geq \frac{(n+ \beta)^2(n-4 -\be)^2}{16}\int_{\B_n} \frac{|f|^2}{\rho^{\be + 4}} dV_g\notag\\
&\quad + 3 C\frac{(n-1)(n-2)(n+\be)(n-4-\be)}4\int_{\B_n} \frac{|f|^2}{\rho^{\be+2}} dx,
\end{align}
for any $f \in C_0^\infty(\B_n)$.
\end{theorem}

By iterating the inequalities \eqref{eq:KO}, \eqref{eq:improveR} and \eqref{eq:hyper}, we obtain the following improved Rellich type inequalities of higher order derivatives in $\mathbb H^n$.

\begin{theorem}\label{eq:higherorder}
Suppose $n\geq 3$ and $k \in (1,n/2)$ be an integer, and $-2 < \beta < n-2k$. Then there exists a constant $C>0$ such that the following inequalities hold for any $f \in C_0^\infty(\B_n)$
\begin{description}
\item (i) If $k = 2l$, $l\geq 1$ then we have
\begin{align}\label{eq:quantievena}
\int_{\B_n} \frac{|\De_g^l f|^2 }{\rho(x)^\beta} dV_g &\geq \frac1{c(n,2l,\beta,2)}\int_{\B_n} \frac{|f|^2}{\rho(x)^{\beta + 4l}} dV_g \notag\\
&\quad + \frac{6C(n-1)}{(n-4l-\beta)c(n,2l,\beta,2)} \int_{\B_n} \frac{|f|^2}{\rho^{4l+\be -2}} dx.
\end{align}
\item (ii) If $k = 2l+1$, $l\geq 1$ then we have
\begin{align}\label{eq:quantiodda}
&\frac{4}{(n-2-\beta)^2}c(n,2l,2+\beta,2)\int_{\B_n} \frac{|\na_g\De_g^l f|_g^2 }{\rho(x)^\beta} dV_g \notag\\
&\quad \geq \int_{\B_n} \frac{|f|^2}{\rho^{\beta +2(2l+1)}} dV_g  + \frac{6C(n-1)}{(n-2(2l+1)-\beta)} \int_{\B_n} \frac{|f|^2}{\rho^{\beta + 4l}} dx.
\end{align}
\end{description}
\end{theorem}

By the same way, we can obtain the critical Rellich type inequalities in $\mathbb H^n$ for $\De_g^l$ and $\na_g \De_g^l$ and their improvements. The details are left for interest readers.


\section*{Acknowledgments}
The author would like to thank Professor Alexandru Krist\'aly for drawing our attentions to his works \cite{FKV,K1,K2}. This work was supported by the CIMI's postdoctoral research fellowship.

\end{document}